\documentclass[11pt]{amsart}
\usepackage[letterpaper,margin=1in]{geometry} 
\usepackage[colorlinks,
             linkcolor=blue,
             citecolor=green,
             pdfproducer={pdfLaTeX},
             pdfpagemode=None,
             bookmarksopen=true
             bookmarksnumbered=true]{hyperref}
\usepackage{parskip}
\usepackage[matrix,arrow,ps,color,line,curve,frame,all]{xy} 

\title[Reflexivity and dualizability in categorified linear algebra]{Reflexivity and dualizability in categorified\\ linear algebra}

\author{Martin Brandenburg}
\thanks{M. Brandenburg: \MAILTO{brandenburg@uni-muenster.de}}

\author{Alexandru Chirvasitu}
\thanks{A. Chirvasitu: \MAILTO{chirva@math.washington.edu}}

\author{Theo Johnson-Freyd}
\thanks{T. Johnson-Freyd: \MAILTO{theojf@math.northwestern.edu}}

\usepackage{amsmath,amsthm,stmaryrd}
\usepackage{cleveref}

\newtheorem{theorem}[subsection]{Theorem}
\newtheorem{lemma}[subsection]{Lemma}
\newtheorem{proposition}[subsection]{Proposition}
\newtheorem{corollary}[subsection]{Corollary}

\theoremstyle{definition} 

\newtheorem{definition}[subsection]{Definition}
\newtheorem{examplenodiamond}[subsection]{Example}
\newtheorem{remarknodiamond}[subsection]{Remark}

\newenvironment{example}{\begin{examplenodiamond}}{\hfill\ensuremath\Diamond\end{examplenodiamond}}
\newenvironment{remark}{\begin{remarknodiamond}}{\hfill\ensuremath\Diamond\end{remarknodiamond}}

\renewcommand\qedhere{\qed}

\newcounter{stepofproof}

\newenvironment{proofstep}[1]{\refstepcounter{stepofproof}\textbf{Step \arabic{stepofproof}: #1}}{}

\crefname{section}{Section}{Section}
\crefformat{section}{#2Section~#1#3} 
\Crefformat{section}{#2Section~#1#3}

\crefname{definition}{Definition}{Definitions}
\crefformat{definition}{#2Definition~#1#3} 
\Crefformat{definition}{#2Definition~#1#3} 

\crefname{example}{Example}{Examples}
\crefformat{example}{#2Example~#1#3} 
\Crefformat{example}{#2Example~#1#3} 

\crefname{examplenodiamond}{Example}{Examples}
\crefformat{examplenodiamond}{#2Example~#1#3} 
\Crefformat{examplenodiamond}{#2Example~#1#3} 

\crefname{remark}{Remark}{Remarks}
\crefformat{remark}{#2Remark~#1#3} 
\Crefformat{remark}{#2Remark~#1#3} 

\crefname{remarknodiamond}{Remark}{Remarks}
\crefformat{remarknodiamond}{#2Remark~#1#3} 
\Crefformat{remarknodiamond}{#2Remark~#1#3} 

\crefname{convention}{Convention}{Conventions}
\crefformat{convention}{#2Convention~#1#3} 
\Crefformat{convention}{#2Convention~#1#3} 

\crefname{lemma}{Lemma}{Lemmas}
\crefformat{lemma}{#2Lemma~#1#3} 
\Crefformat{lemma}{#2Lemma~#1#3} 

\crefname{proposition}{Proposition}{Propositions}
\crefformat{proposition}{#2Proposition~#1#3} 
\Crefformat{proposition}{#2Proposition~#1#3} 

\crefname{corollary}{Corollary}{Corollaries}
\crefformat{corollary}{#2Corollary~#1#3} 
\Crefformat{corollary}{#2Corollary~#1#3} 

\crefname{theorem}{Theorem}{Theorems}
\crefformat{theorem}{#2Theorem~#1#3} 
\Crefformat{theorem}{#2Theorem~#1#3} 

\crefname{assumption}{Assumption}{Assumptions}
\crefformat{assumption}{#2Assumption~#1#3} 
\Crefformat{assumption}{#2Assumption~#1#3} 

\crefname{equation}{}{}
\crefformat{equation}{(#2#1#3)} 
\Crefformat{equation}{(#2#1#3)}

\crefname{proofstep}{Step}{Steps}
\crefformat{proofstep}{#2Step~#1#3} 
\Crefformat{proofstep}{#2Step~#1#3}

\usepackage{tikz}
\usetikzlibrary{arrows}

\tikzset{
    state/.style={
           rectangle,
           rounded corners,
           draw=black, 
           minimum height=2em,
           inner sep=4pt,
           text centered,
           },
}

\newcommand\cat[1]{\textsc{#1}}
\newcommand\define[1]{\emph{#1}}

\newcommand\arXiv[1]{\href{http://arxiv.org/abs/#1}{\nolinkurl{arXiv:#1}}}
\newcommand\MRnumber[1]{\href{http://www.ams.org/mathscinet-getitem?mr=#1}{\nolinkurl{MR#1}}}
\newcommand\DOI[1]{\href{http://dx.doi.org/#1}{\nolinkurl{DOI:#1}}}
\newcommand\MAILTO[1]{\href{mailto:#1}{\nolinkurl{#1}}}

\usepackage{amsfonts,amssymb}

\newcommand\bG{\mathbb G}

\newcommand\bK{\mathbb K}
\newcommand\bL{\mathbb L}

\newcommand\bN{\mathbb N}

\newcommand\bP{\mathbb P}

\newcommand\cC{\mathcal C}
\newcommand\cD{\mathcal D}
\newcommand\cE{\mathcal E}

\newcommand\cI{\mathcal I}
\newcommand\cJ{\mathcal J}

\newcommand\cM{\mathcal M}

\newcommand\cO{\mathcal O}

\newcommand\cS{\mathcal S}
\newcommand\cT{\mathcal T}

\newcommand\cX{\mathcal X}
\newcommand\cY{\mathcal Y}

\DeclareMathOperator\Hom{\cat{Hom}}
\DeclareMathOperator\End{\cat{End}}
\newcommand\Vect{\cat{Vect}}
\newcommand\Pres{\cat{Pres}}
\DeclareMathOperator\Spec{Spec}

\DeclareMathOperator\QCoh{\cat{Qcoh}}
\newcommand\op{\mathrm{op}}
\DeclareMathOperator\Fun{\cat{Fun}}
\newcommand\dd{\mathrm{dd}}
\newcommand\ev{\mathrm{ev}}
\newcommand\coev{\mathrm{coev}}
\newcommand\can{\mathrm{can}}
\DeclareMathOperator*\colim{colim}

\renewcommand\lim{\varprojlim}
\renewcommand\colim{\varinjlim}

\begin{document}

\maketitle

\begin{abstract}
  The ``linear dual'' of a cocomplete linear category  $\cC$ is the category of all cocontinuous linear functors $\cC \to \Vect$.  We study the  questions of when a cocomplete linear category is reflexive (equivalent to its double dual) or dualizable (the pairing with its dual comes with a corresponding copairing).  Our main results are that the category of comodules for a countable-dimensional coassociative coalgebra is always reflexive, but (without any dimension hypothesis) dualizable if and only if it has enough projectives, which rarely happens.  Along the way, we prove that the category $\QCoh(X)$ of quasi-coherent sheaves on a stack $X$ is not dualizable if $X$ 
  is the classifying stack of a semisimple algebraic group in positive characteristic or if $X$ is a scheme containing a closed projective subscheme of positive dimension,
   but is dualizable if $X$ is the quotient of an affine scheme by a virtually linearly reductive group. Finally we prove
   tensoriality (a type of Tannakian duality)
   for affine ind-schemes with countable indexing poset.
\end{abstract}

\section{Introduction}

Fix a field $\bK$.  For cocomplete $\bK$-linear categories $\cC$ and $\cD$, let $\Hom(\cC,\cD) = \Hom_{c,\bK}(\cC,\cD)$ denote the category of cocontinuous $\bK$-linear functors and natural transformations from $\cC$ to $\cD$, and let $\cC\boxtimes\cD = \cC \boxtimes_{c,\bK} \cD$ denote the universal cocomplete $\bK$-linear category receiving a functor $\cC \times \cD \to \cC \boxtimes \cD$ which is cocontinuous and $\bK$-linear in each variable (while holding the other variable fixed).  
If $\cC$ and $\cD$ are both locally presentable (see \Cref{defn.lpcat}), then $\Hom(\cC,\cD)$ and $\cC \boxtimes \cD$ both exist as locally small categories.  Moreover, both are locally presentable  $\bK$-linear categories, and $\Hom$ and $\boxtimes$ satisfy a hom-tensor adjunction.
  The unit for $\boxtimes$ is $\Vect = \Vect_\bK$, and we will let $\cC^* = \Hom(\cC,\Vect)$. It is the ``linear dual'' of $\cC$ in the context of cocomplete $\bK$-linear categories.

Thus the bicategory of locally presentable $\bK$-linear categories provides one possible categorification of linear algebra.  It includes as a full sub-bicategory the Morita bicategory $\cat{Alg}$ of (associative) algebras, bimodules, and intertwiners; the inclusion sends an algebra $A$ to the category $\cM_A$ of right $A$-modules.  But locally presentable $\bK$-linear categories are more general: among them are, for any scheme $X$ over $\bK$, the category $\QCoh(X)$ of quasi-coherent sheaves of $\cO_X$-modules, and also for any (coassociative) coalgebra $C$ over $\bK$, the category $\cM^C$ of right $C$-comodules.

In any generalization of linear algebra, it is interesting to ask what  are the ``finite-dimensional'' objects.  Two possible generalizations of finite-dimensionality are:

\begin{definition}\label{defn.dualizable}

 A locally presentable $\bK$-linear category $\cC$ is called \define{dualizable} if the canonical functor $\cC \boxtimes \cC^* \to \End(\cC)$ is an equivalence.    
A locally presentable $\bK$-linear category $\cC$ is called \define{reflexive} if the canonical functor $\cC \to (\cC^*)^*$ is an equivalence. 

\end{definition}

Dualizability implies reflexivity, but the converse is generally false.  Dualizability is particularly important in light of \cite{BaeDol95,Lur09}.  For example, 
a well-known corollary of the Eilenberg--Watts theorem (which asserts that $\Hom(\cM_{A},\cM_{B}) \simeq {\cM_{A^{\op}\otimes B}}$) answers the question of dualizability in the affirmative for objects of the Morita bicategory $\cat{Alg}$:

\begin{theorem}[Folklore] \label{thm.EW}

  For any associative algebra $A$, the category $\cM_A$ is dualizable. The dual is $(\cM_A)^* \simeq \cM_{A^{\op}}$. \qedhere

\end{theorem}

There is no Eilenberg--Watts theorem for coassociative coalgebras.  Thus the questions of reflexivity and dualizability are more subtle.  Our main results on dualizability are:

\begin{theorem}  \label{thm.coalgebras-dualizability}
   Let $C$ be a coassociative coalgebra.  Then $\cM^C$ is dualizable if and only if it has enough projectives.
\end{theorem}
Such a coalgebra is called \define{right semiperfect} in~\cite{MR0498663} (generalizing the notion from~\cite{MR0157984}).
\begin{theorem} \label{thm.projective}
  Let $X$ be a $\bK$-scheme. If $X$ has a closed projective subscheme of positive dimension, then $\QCoh(X)$ is not dualizable. 
\end{theorem}

In order to state the next result, recall (\cite{MR1395068}) that a linear algebraic group is \define{linearly reductive} if its category of representations is semisimple. It is \define{virtually linearly reductive} if it has a linearly reductive normal algebraic subgroup such that the quotient is a finite group scheme.

\begin{theorem}\label{thm.stacks}
  Let $X$ an affine scheme over $\bK$ and $G$ a virtually linearly reductive group over $\bK$ acting on $X$.  Let $[X/G]$ denote the corresponding quotient stack.  Then $\QCoh([X/G])$ is dualizable.
\end{theorem}

\Cref{thm.coalgebras-dualizability,thm.projective,thm.stacks} are proved in \Cref{section.dualizability}.  (\Cref{section.pres} reviews some of the theory of locally presentable categories.)
In \cref{thm.stacks}, 
the condition on $G$ is important.  Indeed, according to~\cite{MR1395068}, corrected in~\cite{MR1683266}, $\QCoh([\Spec(\bK)/G]) \simeq \cat{Rep}(G) \simeq \cM^{\cO(G)}$ contains a non-zero projective if and only if $G$ is virtually linearly reductive, and moreover if any non-zero injective fails to be projective, then there are no non-zero projectives.  In particular, with \Cref{thm.coalgebras-dualizability} this implies that $\QCoh([\Spec(\bK)/G])$ is not dualizable for $G$ a semisimple group in positive characteristic, nor is it dualizable for $G = \bG_a = \Spec(\bK[x])$.  Such nondualizability results are in stark contrast with~\cite{MR2669705}, where it is shown that for many stacks of geometric interest, the corresponding \emph{derived} category of quasi-coherent sheaves \emph{is} dualizable (for the derived version of~$\boxtimes$).

We address reflexivity in \Cref{section.reflexivity}.  Our main result there is:

\begin{theorem}\label{thm.reflexivity-stronger}
  Suppose that $A = (\dots \twoheadrightarrow A_{2} \twoheadrightarrow A_{1} \twoheadrightarrow A_{0})$ is an $\bN$-indexed projective system of associative $\bK$-algebras.  Let $\cC$ denote the category of injective systems $(M_{0}\hookrightarrow M_{1}\hookrightarrow M_{2} \hookrightarrow \dots)$ where each $M_{i}$ is a right $A_{i}$-module and each inclusion $M_{i} \hookrightarrow M_{i+1}$ identifies $M_{i}$ as the maximal $A_{i}$-submodule of $M_{i+1}$.  Let $B$ be any associative $\bK$-algebras.  Then the ``double dual'' functor
  $$ \dd : \cC \boxtimes \cM_{B} \to \Hom(\cC^{*},\cM_{B}), $$
 which maps $X \boxtimes M$ to $(F \mapsto F(X) \otimes M)$, is an equivalence.
\end{theorem}

An immediate corollary establishes reflexivity for some categories that are not dualizable:

\begin{corollary} \label{thm.coalgebras-reflexivity}
  Let $C$ be a countable-dimensional 
  coassociative coalgebra.  Then $\cM^C$ is reflexive. 
  \end{corollary}

We do not know if countable dimensionality can be dropped.

\begin{proof}
  By \cite[Theorem 2.2.1]{MR0252485}, there exists an increasing sequence of finite-dimensional subcoalgebras $C_{0}\subseteq C_{1} \subseteq \dots$ with $C = \bigcup C_{i}$.  Let $A_{i} = C_{i}^{*}$ and $\cC$ as in \cref{thm.reflexivity-stronger}.  Then $\cC \simeq \cM^{C}$ by \cref{eg.coalgebra}.   \Cref{thm.reflexivity-stronger} with $B = \bK$ completes the proof.
\end{proof}

We end by proving the following symmetric monoidal version of \Cref{thm.reflexivity-stronger}:

\begin{theorem} \label{thm.tensoriality}
  Affine ind-schemes indexed by countable posets are tensorial in the sense of~\cite{Brandenburg2011}.
\end{theorem}

We will recall the definition of tensoriality in \Cref{defn.tensoriality}. Tensoriality of ind-schemes is not covered by~\cite{MR3144607,HallRydh2014}, where finite presentability is crucial.


\section{Recollections on locally presentable categories} \label{section.pres}

In this paper the word ``category'' always means ``locally small category'' unless otherwise specified: the objects might form a proper class, but the morphisms between any two objects should form a set. 

Given two categories $\cX$ and $\cY$, the collection $\Fun(\cX,\cY)$ might fail to be locally small if $\cX$ is not small.  We nevertheless want such collections to be the morphism spaces of a bicategory $\cat{Cat}$, and so we will use the term ``bicategory'' for bicategories whose morphism categories are not locally small.  If $\cX$ and $\cY$ are both $\bK$-linear, we let $\Fun_{\bK}(\cX,\cY)$ denote the collection of $\bK$-linear functors $\cX \to \cY$.

We begin by recalling a few basic facts about locally presentable $\bK$-linear categories.  The primary reference on locally presentable categories is~\cite{MR1294136}.  Many of the results in this section appear as early as~\cite{MR0327863}.

\begin{definition}\label{defn.lpcat}
  For a cardinal $\lambda$ (possibly finite), a partially ordered set $\cI$ is \define{$\lambda$-directed} if any subset of $\cI$ of cardinality strictly less than $\lambda$ has an upper bound.  A \define{$\lambda$-directed colimit} is a colimit indexed by a $\lambda$-directed partially ordered set.  An object $X$ in a category $\cC$ is \define{$\lambda$-little} if $\hom(X,-)$ commutes with $\lambda$-directed colimits, and \define{little} if it is $\lambda$-little for some $\lambda$.  A colimit is \define{$\lambda$-small} if its indexing diagram has strictly fewer than $\lambda$ arrows (including identity arrows).
  
  A set of objects $\Gamma$ in $\cC$ is \define{spanning} if every object of $\cC$ is a colimit of a diagram with objects in $\Gamma$.
  
  A $\bK$-linear category $\cC$ is \define{locally presentable} if it is cocomplete (i.e.\ has all colimits)
  and admits a spanning set consisting entirely of little objects. We may imagine $\cC$ as a ``categorified vector space over $\bK$.''
  
  Locally presentable $\bK$-linear categories are the objects of a bicategory $\Pres_\bK$ whose 1-morphisms are cocontinuous $\bK$-linear functors and whose 2-morphisms are natural transformations.  The special adjoint functor theorem implies that every cocontinuous functor $\cC \to \cD$ with $\cC$ locally presentable has a right adjoint (in the bicategory of all functors).  
\end{definition}

What we call ``$\lambda$-little'' is also called ``$\lambda$-compact'' or ``$\lambda$-presentable'' in the literature.  We will  use the term \define{compact} for $\aleph_0$-little, and \define{tiny} for $2$-little.
Note that for $\lambda$ infinite, the forgetful functor $\Vect \to \cat{Set}$ preserves $\lambda$-directed colimits, so there is no difference between $\cat{Set}$- and $\Vect$-valued homs for the purposes of defining notions like $\lambda$-little.  For $\lambda <\infty$, we use the $\Vect$-valued $\hom$ when deciding if an object is $\lambda$-little.
In particular, an object $X\in \cC$ is tiny if and only if $\hom(X,-) : \cC \to \Vect$ is cocontinuous.  Since arbitrary colimits are composed of direct sums, cokernels, and $\aleph_{0}$-directed colimits, and since direct sums are preserved by all linear functors, $X$ is tiny if and only if it is compact and  $\hom(X,-)$ is right exact.  In an abelian category, the latter is equivalent to $X$ being {projective}.

A set of objects $\Gamma$ in $\cC$ is \define{generating} if $\prod_{X\in \Gamma}\hom(X,-)$ is faithful.  In an abelian category, generating sets are spanning, but this may fail in other categories: the category of two-term filtered vector spaces $(V_{0}\hookrightarrow V_{1})$ is locally presentable, but $(\{0\} \hookrightarrow \bK )$ is generating but not spanning.

The fundamental theorem of locally presentable categories is:

\begin{proposition} \label{thm.fundamental}
  Let $\cC$ be a locally presentable $\bK$-linear category with a spanning set consisting of $\lambda$-little objects, for $\lambda$ a regular cardinal.  Then the full subcategory $\cC_{<\lambda} \subseteq \cC$ of $\lambda$-little objects is essentially small and closed in $\cC$ under $\lambda$-small colimits. Moreover, 

\begin{equation*}\label{eq.fundamental}
\begin{tikzpicture}[auto,baseline=(current  bounding  box.center),shorten >=3pt]
  \node (1) at (0,0) {$\cC$};
  \node[state,text width=4.3cm] (2) at (10,0) {$\bK$-linear functors \mbox{$\cC^\op_{<\lambda}\to \Vect$} that take $\lambda$-small limits to colimits};
  \draw[->] (1) to node{``Yoneda'' functor} node[below]{$Y\mapsto$ restriction of $\hom(-,Y)$ to $\cC_{<\lambda}$ } (2);
\end{tikzpicture}
\end{equation*}
is an equivalence, as is 
\begin{equation*}
\begin{tikzpicture}[auto,baseline=(current  bounding  box.center),shorten >=3pt,shorten <=2pt]
 \node[state,text width=3.7cm] (1) at (0,0) {cocontinuous $\bK$-linear functors $\cC\to\cD$};
 \node[state,text width=3.7cm] (2) at (10,0) {$\bK$-linear functors $\cC_{<\lambda}\to \cD$ preserving $\lambda$-small colimits};
 \draw[->] (1) to node{restriction to $\cC_{<\lambda}$} (2);
\end{tikzpicture}
\end{equation*}
for any cocomplete $\bK$-linear category $\cD$. 
\end{proposition}

Recall that a cardinal $\lambda$ is \define{regular} if it is not the union of strictly fewer than $\lambda$ sets, each of which has cardinality strictly less than $\lambda$.  See~\cite[Exercise 1.b]{MR1294136} for why one may assume $\lambda$ to be regular in \Cref{defn.lpcat}.

\begin{proof}
  The statement consists essentially of Theorem 1.46 and Proposition 1.45 from~\cite{MR1294136} (an error in the proof was corrected in~\cite{MR1699341}), with two changes: we present a $\bK$-linear version, and we allow $\lambda$ to be finite.  To incorporate linearity is straightforward: one can systematically develop a $\bK$-linear theory exactly parallel to the usual theory with no surprises, making only a few changes as necessary; see~\cite{MR648793}.  (Indeed, one can systematically develop a version for categories enriched in any closed symmetric monoidal locally presentable category, although at that level of generality one must use weighted colimits and not just colimits.)
  
  In this paper we will only need  \Cref{thm.fundamental} when $\lambda$ is finite.  
  The only finite regular cardinals are $0$, $1$, and $2$, and $0$-, $1$-, and $2$-little objects are the same: they are the {tiny} objects, i.e.\ the objects $X$ for which $\hom(X,-) : \cC \to \Vect$ is cocontinuous.  It follows from \cite[Remark 1.19]{MR1294136} that if $\cC$ is spanned by its tiny objects, then the full subcategory of tiny objects is essentially small.
  
  Let $\Gamma$ be any set of tiny objects in $\cC$, and abusing notation let $\Gamma$ also denote the full subcategory of $\cC$ on those objects.
    We consider the Yoneda functor $\gamma: \cC \to \Fun_{\bK}(\Gamma^{\op},\Vect)$ sending $Y \in \cC$ to the restriction of $\hom(-,Y)$ to $\Gamma$.  
  
  Since $\Gamma$ consists entirely of tiny objects and since colimits in $\Fun_{\bK}(\Gamma^{\op},\Vect)$ are computed pointwise, $\gamma$ is cocontinuous.  The enriched Yoneda Lemma~\cite{MR2177301} implies that $\Fun_{\bK}(\Gamma^{\op},\Vect)$ is the universal $\bK$-linear cocompletion of $\Gamma$: for any cocomplete $\bK$-linear category $\cD$, the restriction along the Yoneda embedding $\Gamma \hookrightarrow \Fun_{\bK}(\Gamma^{\op},\Vect)$ establishes an equivalence between $\Fun_{\bK}(\Gamma,\cD)$ and $\Hom\bigl(\Fun_{\bK}(\Gamma^{\op},\Vect),\cD\bigr)$.  Let $\iota : \Fun_{\bK}(\Gamma^{\op},\Vect) \to \cC$ denote the cocontinuous $\bK$-linear functor corresponding to the embedding $\Gamma \hookrightarrow \cC$.  Then $\gamma \circ \iota$ restricts to (something canonically isomorphic to) the identity on $\Gamma$, and hence (being cocontinuous and $\bK$-linear) is the identity on $\Fun_{\bK}(\Gamma^{\op},\Vect)$.  It follows that $\gamma$ is essentially surjective.
  
  We suppose now that $\Gamma$ is spanning.  For $Y$ and $Y'$ in $\cC$, we can therefore choose diagrams $\{X_{i}\}_{i\in \cI}$ and $\{X'_{j}\}_{j\in \cJ}$ in $\Gamma$ with $\colim_{i\in \cI} X_{i} \cong Y$ and $\colim_{j\in \cJ} X'_{j} \cong Y'$.  Then there are canonical isomorphisms
  \begin{equation*}
   \hom\bigl(Y,Y'\bigr) \cong \hom\Bigl( \colim_{i\in \cI} X_{i}, \colim_{j\in \cJ} X'_{j} \Bigr) \\ \cong \lim_{i\in \cI}\hom\Bigl(  X_{i}, \colim_{j\in \cJ} X'_{j} \Bigr) \cong \lim_{i\in \cI} \colim_{j\in \cJ}\hom\bigl(  X_{i}, X'_{j} \bigr). \end{equation*}
  In the last step, we use that every $X_{i}$ is tiny, allowing the colimit to commute with $\hom(X_{i},-)$.  On the other hand, since $\gamma$ is cocontinuous,
  \begin{multline*}
   \hom\bigl(\gamma(Y),\gamma(Y')\bigr) \cong \hom\Bigl( \colim_{i\in \cI} \gamma(X_{i}), \colim_{j\in \cJ} \gamma(X'_{j}) \Bigr) \\ \cong \lim_{i\in \cI} \colim_{j\in \cJ}\hom\bigl(  \gamma(X_{i}), \gamma(X'_{j}) \bigr) \cong \lim_{i\in \cI} \colim_{j\in \cJ}\hom\bigl(  X_{i}, X'_{j} \bigr). \end{multline*}
   as colimits in   $\Fun_{\bK}(\Gamma^{\op},\Vect)$ are computed pointwise and $\gamma(X) = \hom(-,X)$ for $X\in \Gamma$.  Thus $\gamma$ is full and faithful.
\end{proof}

 An important corollary of \cref{thm.fundamental} is that $\Hom_{\Pres_{\bK}}(\cC,\cD)$ is locally small, being equivalent to a full subcategory of the category of functors from a small category to a locally small category.

\begin{remark}
  A converse to the finite-$\lambda$ portion of \Cref{thm.fundamental} is the observation that, for any small $\bK$-linear category $\Gamma$, in $\Fun_{\bK}(\Gamma^\op,\Vect)$ every representable functor $\hom(-,X)$ for $X\in \Gamma$ is tiny, and the Yoneda Lemma implies that the representable functors are a spanning set.  Thus \Cref{thm.fundamental} identifies categories of the form $\Fun_{\bK}(\Gamma^{\op},\Vect)$ as precisely the tinily-spanned  $\bK$-linear categories.
\end{remark}

\begin{lemma}[{\cite[Exercise 2.n]{MR1294136}, \cite[Proposition 2.1.11]{MR3097055}}] \label{lemma.2limitsandcolimits}
  The bicategory $\Pres_\bK$ of locally presentable $\bK$-linear categories is closed under all small 2-limits and all small 2-colimits:  
\begin{itemize}
\item To compute a 2-limit in $\Pres_\bK$, simply compute the same 2-limit in the bicategory $\cat{Cat}$ of categories, ignoring that the arrows happen to be left adjoints.  
\item To compute a 2-colimit in $\Pres_\bK$, replace every 1-morphism by its right adjoint --- this is a contravariant bifunctor $\Pres_\bK \to \cat{Cat}$ --- and compute the corresponding 2-limit in $\cat{Cat}$, ignoring that the arrows happen to be right adjoints.\qedhere 
\end{itemize}
\end{lemma}


\begin{remark} \label{rem.lim}
We will be primarily interested in (2-)limits and (2-)colimits indexed by partially ordered sets $\cI$.  For a diagram $\cC : \cI \to \Pres_\bK$, $i \mapsto \cC_i$,  \Cref{lemma.2limitsandcolimits} says that $\varinjlim_{i \in \cI} \cC_i$ consists of  colimits of the form $\varinjlim_{i\in \cI} X_i$ for $X_i \in \cC_i$, and $\varprojlim \cC_i$ consists of  limits $\varprojlim X_i$, where to make sense of the map $X_i \to X_j$ for $i < j$ one must involve the functor $\cC_i \to \cC_j$.

One corollary of \cref{lemma.2limitsandcolimits} is that if $\cI$ is just a set then products and coproducts in $\Pres_\bK$ indexed by $\cI$ agree (and agree with the product, but not the coproduct, of underlying categories).  It thus makes sense to call this (co)product the \define{direct sum} $\bigoplus_{i\in \cI}\cC_i$.
\end{remark}

Just like vector spaces, in addition to a direct sum, locally presentable $\bK$-linear categories also admit a tensor product:

\begin{lemma}[{\cite[Exercise 1.l]{MR1294136}, \cite[Corollary 2.2.5]{MR3097055}}]\label{lemma.homtensoradjunction} \mbox{}
\begin{enumerate}
  \item  For $\cC,\cD \in \Pres_\bK$, the category $\Hom(\cC,\cD)$ is also locally presentable and $\bK$-linear.  

  \item  For each $\cC$, the bifunctor $\Hom(\cC,-) : \Pres_\bK \to \Pres_\bK$ has a left adjoint $(-)\boxtimes \cC$, making $\Pres_\bK$ into a closed symmetric monoidal bicategory.  

  \item The \define{tensor product} $\cC \boxtimes \cD$ satisfies the following universal property: For any  
  cocomplete $\bK$-linear category $\cE$, cocontinuous functors $\cC \boxtimes \cD \to \cE$ are the same as functors $\cC \times \cD \to \cE$ that are cocontinuous and $\bK$-linear in each variable, holding the other variable fixed. \qedhere   
\end{enumerate}
\end{lemma}

\begin{remark}
One can present $\cC \boxtimes \cD$ from part (3) of \Cref{lemma.homtensoradjunction} in a number of ways; for example, it is (up to canonical equivalence) the category of continuous functors $\cC^{\op} \to \cD$.
%
%
%
\end{remark}

\begin{definition}\label{defn.tensor_loc_pres}
We say that $\cS \in \Pres_\bK$ is a \define{symmetric monoidal locally presentable $\bK$-linear category} if it is symmetric monoidal as a category and the monoidal structure $\otimes : \cS \times \cS \to \cS$ is cocontinuous and $\bK$-linear in each variable, so that it extends to a 1-morphism $\cS \boxtimes \cS \to \cS$. 
In fact, $\cS$ is a commutative monoid in the symmetric monoidal bicategory $\Pres_\bK$.
Therefore, we may imagine $\cS$ as a ``categorified commutative algebra over $\bK$''.

We denote by $\Pres_{\otimes,\bK}$ the bicategory of symmetric monoidal locally presentable $\bK$-linear categories, where for two such objects $\cC,\cD$ the category $\Hom_{\otimes,c,\bK}(\cC,\cD)$ consists of cocontinuous  symmetric monoidal $\bK$-linear functors and symmetric monoidal natural transformations.  We will often simply write $\Hom_\otimes$ in place of $\Hom_{\otimes,c,\bK}$, as $\bK$ will often  be implicit and we will never use non-cocontinuous symmetric monoidal functors.

Let $\cS \in \Pres_{\otimes,\bK}$. An \define{$\cS$-module in $\Pres_\bK$} is a locally presentable $\bK$-linear category $\cC$ together with an \define{action} $(-)\triangleright (-): \cS \boxtimes \cC \to \cC$ and unit and associativity data making the appropriate triangles and pentagons commute.
\end{definition}

\begin{remark}\label{remark.module>enriched}
In the locally presentable setting, module categories are in particular enriched: For any object $M\in \cC$, the functor $(-)\triangleright M : \cS \to \cC$ is cocontinuous, hence has a right adjoint $\hom(M,-)_\cS : \cC \to \cS$.  It is straightforward to check that the associativity and unit data make $\hom(-,-)_\cS$ into the data of an $\cS$-enriched-category structure on~$\cC$.
\end{remark}

\begin{remark}\label{remark.basechange}
For a field extension $\bK \to \bL$, an $\bL$-linear locally presentable category is precisely a $\Vect_\bL$-module in $\Pres_\bK$. It follows that for any field extension $\bK\to\bL$ we have a 2-functor
\[
 \boxtimes \Vect_\bL:\Pres_\bK\to\Pres_\bL 
\] 
(tensor product is over $\Vect_\bK$) which preserves $\boxtimes$, i.e.\ it is a symmetric monoidal 2-functor. 
\end{remark}

The following enriched version of tinyness will be useful in the course of the proof of \Cref{thm.stacks}. 

\begin{definition}\label{defn.tinygeneration}
Let $\cS\in\Pres_{\otimes,\bK}$, and $\cC\in\Pres_\bK$ an $\cS$-module in the sense of \Cref{defn.tensor_loc_pres}.  
An object $M \in \cC$ is \define{tiny over $\cS$} if the enriched hom functor $\hom(M,-)_{\cS} : \cC \to \cS$ is cocontinuous. 
We say that $\cC$ is \define{tinily-spanned over $\cS$} if every object is an $\cS$-weighted colimit of tiny-over-$\cS$ objects. 
\end{definition}

\begin{proposition}\label{pr.tinygeneration}
If $\cS\in\Pres_{\otimes,\bK}$ is tinily-spanned over $\Vect$ and $\cC\in\Pres_\bK$ is an $\cS$-module which is tinily-spanned over $\cS$, then $\cC$ is also tinily-spanned over $\Vect$.
\end{proposition}

\begin{remark}\label{remark.tinygeneration}    
$\Vect$ can be replaced throughout by any symmetric monoidal locally presentable category, but we do not need this stronger version.
\end{remark}

\begin{proof}[Proof of \Cref{pr.tinygeneration}]
Let $\{s_i\}$ be a spanning set of tiny objects in $\cS$ and $\{c_j\}$ a set of tiny-over-$\cS$ objects in $\cC$ that spans $\cC$ over $\cS$. 

Recall the action $s\triangleright c$ of an object $s\in \cS$ on an object $c\in\cC$ implicit in the notion of $\cS$-module (\Cref{defn.tensor_loc_pres}). We have
\[
 \hom_\cC(s_i\triangleright c_j,-)\cong \hom_\cS(s_i,\hom_\cC(c_j,-)_\cS):\cC\to\Vect.
\] 
The right hand side is cocontinuous by assumption, so the left hand side is as well. Hence, $s_i\triangleright c_j$ are tiny. The fact that they span follows from the fact that $s_i$ and $c_j$ span and an unpacking of the notion of weighted colimit.  
\end{proof}

We may now introduce our main examples:

\begin{example}\label{eg.modules}
  Let $A$ be an associative algebra over $\bK$.  Then the category $\cM_A$ of all right $A$-modules is a locally presentable $\bK$-linear category.  The spanning set $\Gamma$ may be taken to consist of the rank-one free module $A$, which is tiny.  \Cref{thm.fundamental} implies that for any  cocomplete $\bK$-linear category $\cD$, cocontinuous $\bK$-linear functors $\cM_{A} \to \cD$ are equivalent to $\bK$-linear functors $\Gamma \to \cD$.  It then follows from \Cref{lemma.homtensoradjunction} that there is a canonical equivalence of categories $\Hom(\cM_A,\cD) \simeq {_A\cM} \boxtimes \cD$ for any $\cD \in \Pres_\bK$, where $_A\cM \cong \cM_{A^\op}$ is the category of left $A$-modules.  The Eilenberg--Watts theorem and its corollary \Cref{thm.EW} follow.
\end{example}

\begin{example}\label{eg.coalgebra}
  Let $C$ be a coassociative coalgebra over $\bK$.   By the so-called fundamental theorem of coalgebras (\cite[Theorem 2.2.1]{MR0252485}), $C \simeq \varinjlim_{i\in \cI} C_i$, where $\{C_i\}_{i\in \cI}$ is the partially ordered set of finite-dimensional sub-coalgebras of $C_i$.  Moreover, every right $C$-comodule $X$ is canonically a colimit $X = \varinjlim_{i\in \cI}X_i$ where $X_i$ is the largest subcomodule for which the coaction $X_i \to X_i \otimes C$ factors through $X_i \otimes C_i$.
  
  Give the linear dual $C_i^*$ the algebra structure $(\alpha \cdot \beta)(c) = \sum \alpha(c_{(2)}) \otimes \beta(c_{(1)})$, where the comultiplication on $C_i$ in Sweedler's notation is $c \mapsto c_{(1)} \otimes c_{(2)}$.  Then the category $\cM^{C_i}$ of right $C_i$-comodules is canonically isomorphic to the category $\cM_{C_i^*}$ of right $C_i^*$-modules, hence locally presentable by \Cref{eg.modules}.  The remarks in the previous paragraph then amount to the fact that
  $$ \cM^C \simeq \varinjlim_{i\in \cI}\cM^{C_i}, $$
  where the left-hand side denotes the category of all right $C$-comodules and the right-hand side denotes the colimit in $\Pres_\bK$.  In particular, $\cM^C$ is locally presentable.
\end{example}

\begin{example} \label{eg.scheme}
  Let $X$ be a scheme.  We can present $X$ as a colimit of affines: $X = \varinjlim_{i\in \cI} \Spec(A_i)$.  Let $\QCoh(X)$ denote the category of quasi-coherent sheaves of $\cO_X$-modules.  By definition:
  $$ \QCoh(X) \simeq \varprojlim_{i \in \cI^{\op}} \QCoh(\Spec(A_i)) \simeq \varprojlim_{i\in \cI^\op}{_{A_i}}\cM $$
  A priori, this limit is computed in $\cat{Cat}$.  But the pullback functors ${_{A_i}}\cM \to {_{A_j}}\cM$ involved are cocontinuous, and so by \Cref{lemma.2limitsandcolimits} it is also the limit in $\Pres_\bK$.  In particular, $\QCoh(X)$ is locally presentable.
  
  More generally, any Artin stack can be presented as a 2-colimit of affine schemes, and a similar argument applies.
\end{example}


\section{(Non)dualizability}\label{section.dualizability}

In this section we will prove \Cref{thm.coalgebras-dualizability,thm.projective,thm.stacks}. First, let us briefly unpack the notion of dualizability from \Cref{defn.dualizable}.

We will make use of the notion of adjunction between bicategories, referring the reader to \cite[Chapter 9]{Fio06} for background. The term used in that reference is ``biadjunction'', but we will simply say ``adjunction''. The following Lemma reproduces at the 2-categorical level facts that are essentially standard within symmetric monoidal 1-categories:

\begin{lemma}\label{pr.dualizable}
For a locally presentable $\bK$-linear category $\cC$ the following conditions are equivalent:
\begin{enumerate}
  \item $\cC$ is dualizable in the sense of \Cref{defn.dualizable}.
  \item For any $\cD\in\Pres_{\bK}$ the canonical functor $\can_{\cD} : \cD\boxtimes\cC^*\to\Hom(\cC,\cD)$ is an equivalence.  
  \item If $\ev:\cC^*\boxtimes\cC\to\Vect$ is the standard evaluation pairing, then 
    \begin{equation}\label{eq.counit_CC*}
      \begin{tikzpicture}[auto,baseline=(1.base)]
        \path[anchor=base] (0,0) node (1) {$(\bullet\boxtimes\cC^*)\boxtimes\cC\simeq \bullet\boxtimes(\cC^*\boxtimes\cC)$} +(4,0) node (2) {$\mathrm{id}$};
         \draw[->] (1) to node[]{$\scriptstyle -\boxtimes\ev$} (2);   
      \end{tikzpicture}
    \end{equation} 
is the counit of an adjunction between the 2-endofunctors $\bullet\boxtimes\cC$ and $\bullet\boxtimes\cC^*$ of $\Pres_{\bK}$. 
  \item There is a cocontinuous linear functor $\coev:\Vect\to\cC\boxtimes\cC^*$ (the \define{coevaluation}) such that the two compositions 
\begin{equation}\label{eq.dualizableC}
  \begin{tikzpicture}[baseline=(C.base),shorten >=1pt,auto,node distance=3cm,thick,connect/.style={circle,fill=blue!20,draw,font=\sffamily\small}]  
    \node (C) at (0,0) {$\cC$};

    \node (simeq1) at (.5,0) {$\scriptstyle \simeq$};

    \node (VectC1) at (1,.5) {$\Vect$};
    \node (VectC2) at (1,-.5) {$\cC$};

    \node (CC*C1) at (4,1) {$\cC^{\ }$};
    \node (CC*C2) at (4,0) {$\cC^*$};
    \node (CC*C3) at (4,-1) {$\cC^{\ }$};

    \node (CVect1) at (7,.5) {$\cC$};
    \node (CVect2) at (7,-.5) {$\Vect$};

    \node (simeq2) at (7.5,0) {$\scriptstyle \simeq$};

    \node (C_bis) at (8,0) {$\cC$};

    \node (CC*C_boxtimes_1) at (4,.5) {$\boxtimes^{\ }$};
    \node (CC*C_boxtimes_2) at (4,-.5) {$\boxtimes^{\ }$};

    \node (VectC_boxtimes) at (1,0) {$\boxtimes$};
    \node (CVect_boxtimes) at (7,0) {$\boxtimes$};

    \node[connect,label=110:$\scriptstyle \mathrm{coev}$] (coev) at (3,.5) {};
    \node[connect,label=-70:$\scriptstyle \mathrm{ev}$] (ev) at (4.8,-.5) {};
    
    \draw (VectC1) to[bend left=10] (coev);
    \draw[->] (ev) to[bend right=10] (CVect2);
    \draw[->] (VectC2) to[bend right=10] node[auto,swap] {$\scriptstyle\mathrm{id}$} (CC*C3);
    \draw[->] (CC*C1) to[bend left=10] node[auto] {$\scriptstyle\mathrm{id}$} (CVect1);
    \draw[->] (coev) to [bend left] (CC*C1);
    \draw[->] (coev) to [bend right] (CC*C2);
    \draw (CC*C2) to [bend left] (ev);
    \draw (CC*C3) to [bend right] (ev);
 \end{tikzpicture} 
\end{equation}
and
\begin{equation}\label{eq.dualizableC*}
 \begin{tikzpicture}[baseline=(C*.base),shorten >=1pt,auto,node distance=3cm,thick,connect/.style={circle,fill=blue!20,draw,font=\sffamily\small}]  
    \node (C*) at (0,0) {$\cC^*$};

    \node (simeq1) at (.5,0) {$\scriptstyle \simeq$};

    \node (C*Vect1) at (1,.5) {$\ \cC^*$};
    \node (C*Vect2) at (1,-.5) {$\Vect$};

    \node (C*CC*1) at (4,1) {$\cC^*$};
    \node (C*CC*2) at (4,0) {$\cC^{\ }$};
    \node (C*CC*3) at (4,-1) {$\cC^*$};

    \node (VectC*1) at (7,.5) {$\Vect$};
    \node (VectC*2) at (7,-.5) {$\ \cC^*$};

    \node (simeq2) at (7.5,0) {$\scriptstyle \simeq$};

    \node (C*_bis) at (8,0) {$\cC^*$};

    \node (C*CC*_boxtimes_1) at (4,.5) {$\boxtimes^{\ }$};
    \node (C*CC*_boxtimes_2) at (4,-.5) {$\boxtimes^{\ }$};

    \node (C*Vect_boxtimes) at (1,0) {$\boxtimes$};
    \node (VectC*_boxtimes) at (7,0) {$\boxtimes$};

    \node[connect,label=-110:$\scriptstyle \mathrm{coev}$] (coev) at (3,-.5) {};
    \node[connect,label=70:$\scriptstyle \mathrm{ev}$] (ev) at (4.8,.5) {};
    
    \draw (C*Vect2) to[bend right=10] (coev);
    \draw[->] (ev) to[bend left=10] (VectC*1);
    \draw[->] (C*Vect1) to[bend left=10] node[auto] {$\scriptstyle\mathrm{id}$} (CC*C1);
    \draw[->] (CC*C3) to[bend right=10] node[auto,swap] {$\scriptstyle\mathrm{id}$} (VectC*2);
    \draw[->] (coev) to [bend left] (CC*C2);
    \draw[->] (coev) to [bend right] (CC*C3);
    \draw (CC*C1) to [bend left] (ev);
    \draw (CC*C2) to [bend right] (ev);
 \end{tikzpicture}   
\end{equation}
are naturally isomorphic to the identity.
  \item The identity functor $\mathrm{id}_\cC$ is in the essential image of the canonical functor $\cC\boxtimes\cC^*\to\End(\cC)$.
  \item The 2-endofunctor $\boxtimes\cC$ of $\Pres_\bK$ has a right adjoint of the form $\boxtimes\cC'$ for some $\cC'\in\Pres_\bK$.
\end{enumerate}
\end{lemma}
\begin{proof}
{\bf (1) $\Rightarrow$ (5)}  This is immediate.

{\bf (5) $\Rightarrow$ (4)} An object $x\in\cC\boxtimes\cC^*$ that maps onto $\mathrm{id}_\cC$ through $\cC\boxtimes\cC^*\to\End(\cC)$ induces a left adjoint $\coev:\Vect\to\cC\boxtimes\cC^*$, $\bK^{\oplus\alpha} \mapsto x^{\oplus\alpha}$. We claim that as the name suggests, $\coev$ is a coevaluation in the sense of (4). In order to verify this, we have to show that the functors \Cref{eq.dualizableC,eq.dualizableC*} are (naturally isomorphic to) identities. 

For \Cref{eq.dualizableC} this is simply an unpacking of the fact that the image of $x$ in $\End(\cC)$ is the identity. Indeed, the right-hand half of \Cref{eq.dualizableC} is simply the $\cC$-valued pairing of $\cC\boxtimes\cC^*$ with $\cC$ obtained by first mapping the former into $\End(\cC)$ and then evaluating $\cC$-endofunctors at a given object in $\cC$.

The verification is almost as simple for \Cref{eq.dualizableC*}. The desired isomorphism can be tested against $\cC$ by pairing via $\ev$; in other words, it is enough to show that the composition $\ev\circ(\Cref{eq.dualizableC*}\boxtimes\mathrm{id}_{\cC})$ is naturally isomorphic to $\ev:\cC^*\boxtimes\cC\to\Vect$. A diagram chase shows that this composition is isomorphic to $\ev\circ(\mathrm{id}_{\cC^*}\boxtimes\Cref{eq.dualizableC})$, which in turn is isomorphic to $\ev$ because $\Cref{eq.dualizableC}\cong\mathrm{id}_{\cC}$.

{\bf (4) $\Rightarrow$ (3)} For every $\cD,\cE\in\Pres_{\bK}$, the functors $\ev$ and $\coev$ induce functors
\begin{equation*}\label{eq.CC*_adjunction}
  \begin{tikzpicture}[auto,baseline=(current  bounding  box.center)]
        \path[anchor=base] (0,0) node (1) {$\Hom(\cD\boxtimes\cC,\cE)$} +(5,0) node (2) {$\Hom(\cD,\cE\boxtimes\cC^*),$};
         \draw[->] (1) to[bend left=10] node {$R_{\cD,\cE}$} (2);
         \draw[->] (2) to[bend left=10] node {$L_{\cD,\cE}$} (1);
  \end{tikzpicture}
\end{equation*}
natural in $\cD$ and $\cE$ in the obvious sense. The functor $R_{\cD,\cE}$, for instance, sends $F:\cD\boxtimes\cC\to\cE$ to  
\begin{equation*}
 \begin{tikzpicture}[baseline=(D.base),shorten >=1pt,auto,node distance=3cm,thick,connect/.style={circle,fill=blue!20,draw,font=\sffamily\small}]  
    \node (D) at (0,0) {$\cD$};

    \node (simeq1) at (.5,0) {$\scriptstyle \simeq$};

    \node (DVect1) at (1,.5) {$\cD$};
    \node (DVect2) at (1,-.5) {$\Vect$};

    \node (DCC*1) at (4,1) {$\cD^{\ }$};
    \node (DCC*2) at (4,0) {$\cC^{\ }$};
    \node (DCC*3) at (4,-1) {$\cC^*$};

    \node (EC*1) at (7,.5) {$\cE$};
    \node (EC*2) at (7,-.5) {$\ \cC^*$};

    \node (DCC*_boxtimes_1) at (4,.5) {$\boxtimes^{\ }$};
    \node (DCC*_boxtimes_2) at (4,-.5) {$\boxtimes^{\ }$};

    \node (DVect_boxtimes) at (1,0) {$\boxtimes$};
    \node (EC*_boxtimes) at (7,0) {$\boxtimes$};

    \node[connect,label=-110:$\scriptstyle \text{coev}$] (coev) at (3,-.5) {};
    \node[connect,label=70:$\scriptstyle F$] (F) at (4.8,.5) {};
    
    \draw (DVect2) to[bend right=10] (coev);
    \draw[->] (F) to[bend left=10] (EC*1);
    \draw[->] (DVect1) to[bend left=10] node {$\scriptstyle\mathrm{id}$} (DCC*1);
    \draw[->] (DCC*3) to[bend right=10] node[swap] {$\scriptstyle\mathrm{id}$} (EC*2);
    \draw[->] (coev) to [bend left] (DCC*2);
    \draw[->] (coev) to [bend right] (DCC*3);
    \draw (DCC*1) to [bend left] (F);
    \draw (DCC*2) to [bend right] (F);
 \end{tikzpicture}   ,
\end{equation*}
while $L_{\cD,\cE}$ is defined similarly using $\ev$.

Now, the conditions in (4) imply that $R_{\cD,\cE}$ and $L_{\cD,\cE}$ are mutually inverse, so in particular each $R_{\cD,\cE}$ is an equivalence. Collectively, the $R_{\cD,\cE}$ implement an adjunction between $\bullet\boxtimes\cC$ and $\bullet\boxtimes\cC^*$ (where the former is the left adjoint) as in \cite[Definition 9.8]{Fio06}. The identification of the counit with $\ev$ as in \Cref{eq.counit_CC*} is now easy.

{\bf (3) $\Rightarrow$ (2)} By definition $\Hom(\cC,\bullet)$ is a right adjoint to $\bullet\boxtimes\cC$. By the uniqueness of adjoints between bicategories (e.g.\ \cite[Theorem 9.20]{Fio06}), we can find an equivalence $\eta_\cD:\cD\boxtimes\cC^*\simeq\Hom(\cC,\cD)$, natural in $\cD$, that intertwines the counits of the two adjunctions: the diagram
\begin{equation*}
  \begin{tikzpicture}[,baseline=(current  bounding  box.center),anchor=base,cross line/.style={preaction={draw=white,-,line width=6pt}}]
    \path (0,0) node (1) {$\cD\boxtimes\cC^*\boxtimes \cC$} +(2,1.5) node (2) {$\Hom(\cC,\cD)\boxtimes\cC $} +(4,0) node (3) {$\cD$} +(2,.6) node {$\Downarrow$}; 
    \draw[->] (1) to node[pos=.4,auto] {$\scriptstyle \eta_\cD\boxtimes\mathrm{id}_\cC$} (2);
    \draw[->] (1) to node[auto,swap] {$\scriptstyle \mathrm{id}_\cD\boxtimes\ev$} (3);
    \draw[->] (2) to (3);
  \end{tikzpicture}
\end{equation*} 
commutes up to natural isomorphism for every $\cD\in\Pres_{\bK}$. Moreover, we get a similar commutative diagram if we substitute $\can_\cD$ for $\eta_\cD$. But then, by the universality of the counit of an adjunction (the so-called biuniversality of \cite[Definition 9.4]{Fio06}, or rather its dual), the two functors $\eta_\cD$ and $\can_\cD$ are naturally isomorphic; see e.g.\ \cite[Lemma 9.7]{Fio06}.

{\bf (2) $\Rightarrow$ (1)}  Indeed, \Cref{defn.dualizable} is the particular instance of (2) obtained by taking $\cD=\cC$.

{\bf (3) $\Rightarrow$ (6)} Simply set $\cC'=\cC^*$.

{\bf (6) $\Rightarrow$ (3)} The equivalence
\[  \cC'\simeq \Hom(\Vect,\cC')\simeq \Hom(\cC,\Vect) \]
resulting from the adjunction identifies $\cC'$ with $\cC^*$ in such a way that the counit $\cC'\boxtimes\cC\to\Vect$ gets identified with the evaluation $\ev:\cC^*\boxtimes\cC\to\Vect$.   
\end{proof}

\begin{remark}\label{remark.basechange_preserves_dualizability}
Condition (6) in \Cref{pr.dualizable} makes it clear that the base change 2-functor $\boxtimes\Vect_\bL:\Pres_\bK\to\Pres_\bL$ from \Cref{remark.basechange} preserves dualizability. Indeed, $\boxtimes\left(\cC'\boxtimes\Vect_\bL\right)$ is right adjoint to $\boxtimes\left(\cC\boxtimes\Vect_\bL\right)$ on $\Pres_\bL$ whenever $\boxtimes\cC'$ is right adjoint to $\boxtimes\cC$ on $\Pres_\bK$. 
\end{remark}

\begin{remark} \label{rem.direct_sum_dual}
Let $(\cC_i)_{i \in \cI}$ be a family of locally presentable $\bK$-linear categories. Then its direct sum $\bigoplus_{i \in I} \cC_i$ is dualizable. This follows from $(1) \Leftrightarrow (2)$ in \Cref{pr.dualizable} and \Cref{rem.lim}.
\end{remark}

For future use, we note the following consequence of \Cref{pr.dualizable}.

\begin{corollary}\label{cor.dualizable_alt_criterion}
Suppose that $\iota:\cD\to\cC$ is a cocontinuous $\bK$-linear functor between locally presentable $\bK$-linear categories, and $\pi:\cC\to \cD$ a  cocontinuous $\bK$-linear functor with $\pi\iota\cong \mathrm{id}_{\cD}$. 
If $\cC$ is dualizable, then so is $\cD$.
\end{corollary}
\begin{proof}
This follows from the commutative diagram 
\begin{equation*}
  \begin{tikzpicture}[auto]
    \path[anchor=base] (0,0) node (DD*) {$\cD\boxtimes\cD^*$} (3,0) node (D>D) {$\Hom(\cD,\cD)$} (0,1.5) node (DC*) {$\cD\boxtimes\cC^*$} (3,1.5) node(C>D) {$\Hom(\cC,\cD)$};
     \draw[->] (DC*) -- node[swap] {$\scriptstyle \mathrm{id}_{\cD}\boxtimes \iota^{*}$} (DD*); \draw[->] (C>D) -- node {$\scriptstyle \iota^{*}$} (D>D);
     \draw[->] (DC*.mid -| DC*.east) -- node {$\simeq$} (C>D.mid -| C>D.west);
     \draw[->] (DD*.mid -| DD*.east) -- node {} (D>D.mid -| D>D.west);
  \end{tikzpicture}
\end{equation*}
where the vertical arrows are both given by restriction along $\iota:\cD\to\cC$.

If $\cC$ is dualizable, then by \Cref{pr.dualizable} the upper horizontal functor is an equivalence. The hypotheses imply that $\mathrm{id}_{\cD}$ is in the essential image of the right hand vertical arrow (e.g.\ it is the image of $\pi\in\Hom(\cC,\cD)$), and hence also in the image of the lower horizontal functor. But then the equivalence between (1) and (5) of \Cref{pr.dualizable} applies to prove dualizability. 
\end{proof}

We may now begin establishing certain categories as either dualizable or not.  First,
\Cref{eg.modules} generalizes immediately to all tinily-spanned categories:

\begin{lemma}\label{cor.dualizability}
  Suppose that a locally presentable $\bK$-linear category $\cC$ has a spanning set $\Gamma$ consisting entirely of tiny objects.  Then $\cC$ is dualizable.
\end{lemma}

\begin{proof}
  As in the proof of \Cref{thm.fundamental}, the Yoneda functor establishes an equivalence of categories $\cC \simeq \Fun_{\bK}(\Gamma^{\op},\Vect)$.  Its dual is then $\cC^* = \Hom(\cC,\Vect) \simeq \Fun_{\bK}(\Gamma,\Vect)$, where the pairing $\cC \boxtimes \cC^* \to \Vect$ is computed as a coend over $\Gamma$, and $$\Hom(\Fun_{\bK}(\Gamma^\op,\Vect),\cD) \simeq \Fun_{\bK}(\Gamma,\cD) \simeq \Fun_{\bK}(\Gamma,\Vect)\boxtimes \cD.$$  The composition is nothing but the functor $\Hom(\Fun_{\bK}(\Gamma^\op,\Vect),\cD) \leftarrow \Fun_{\bK}(\Gamma^\op,\Vect)^*\boxtimes \cD$ induced by the pairing.  By taking $\cD = \cC$ we see that $\cC \simeq \Fun_{\bK}(\Gamma^\op,\Vect)$ is dualizable.
\end{proof}

\begin{remark}  
  We know of no dualizable locally presentable category not of this type, and conjecture that dualizability of a locally presentable $\bK$-linear category implies that the category is tinily-spanned.  \Cref{thm.coalgebras-dualizability} implies that there are no counterexamples to this conjecture among categories of the form $\cM^C$ for $C$ a coassociative coalgebra.
\end{remark}

\Cref{thm.stacks} is now a simple consequence of \Cref{pr.tinygeneration}.

\begin{proof}[Proof of \Cref{thm.stacks}]
Let $\cO(X)$ be the $\bK$-algebra of regular functions on $X$. It carries an action of $G$ that is compatible with the multiplication; i.e.\ $\cO(X)$ makes sense as an algebra object in $\cat{Rep}(G) = \cM^{\cO(G)}$, the symmetric monoidal category of comodules over the Hopf algebra $\cO(G)$ of regular functions on $G$. 

Since $X$ is affine, $\QCoh([X/G])$ is the category of $\cO(X)$-modules with compatible $G$-action, which is to say it is the category of $\cO(X)$-module-objects in the symmetric monoidal category $\cat{Rep}(G)$. In particular, $\QCoh([X/G])$ is a $\cat{Rep}(G)$-module in the sense of \Cref{defn.tensor_loc_pres}. 

We can now apply \Cref{pr.tinygeneration} to the present situation, where we let $\cC=\QCoh([X/G])$ and $\cS=\cat{Rep}(G)$, to conclude that $\cC$ is tinily-spanned over $\Vect$: 

First, $\cO(X)$ is a tiny-over-$\cat{Rep}(G)$ spanning-over-$\cat{Rep}(G)$ object in $\QCoh([X/G])$. Secondly, $\cat{Rep}(G)$ is tinily-spanned over $\Vect$ essentially by \cite{MR1395068}. One result that is part of the main theorem of that paper is that being virtually linearly reductive is equivalent to the monoidal unit $\bK$ in $\cat{Rep}(G)$ having finite-dimensional injective envelope $E$. It is easy to see then that $E^*$ is projective in $\cat{Rep}(G)$ (and hence tiny because it is also finite-dimensional). 

For every finite-dimensional object $V\in\cat{Rep}(G)$ the endofunctor $-\otimes V$ of $\cat{Rep}(G)$ is left adjoint to the exact functor $-\otimes V^*$, and hence preserves projectivity. Since $\bK$ is surjected upon by $E^*$, $V$ is surjected upon by the tiny object $E^*\otimes V$. Since objects in the comodule category $\cat{Rep}(G)$ are unions of finite-dimensional subobjects, the conclusion that $\cat{Rep}(G)$ is tinily-spanned follows.      

Finally, dualizability of $\QCoh([X/G])$ is a consequence of \Cref{cor.dualizability}.
\end{proof}

We turn now to our main categories of interest, namely $\cM^C$ for $C$ a coassociative coalgebra.  Since $\cM^C$ is abelian and in abelian categories generating sets are spanning, \Cref{cor.dualizability} implies that $\cM^C$ is dualizable if it is generated by its compact projectives.  This, of course, fails in general, as the following well-known example illustrates:

\begin{example}
Consider the coalgebra $\bK[x]$ with comultiplication $x^n \mapsto \sum_{i=0}^n x^i \otimes x^{n-i}$.  We claim that in the category $\cM^{\bK[x]}$, there are no nonzero projectives.  Indeed, suppose that $V \in \cM^{\bK[x]}$ is not the zero object.  It suffices to witness a surjection $W \twoheadrightarrow V$ that does not split.  
  
  The finite-dimensional subcoalgebras of the $\bK[x]$ are dual to the algebras $\bK[t]/(t^n)$; thus $\cM^{\bK[x]}$ is the category whose objects are vector spaces $V$ equipped with a locally nilpotent endomorphism~$t$.  In particular, $\ker t$ is not zero if $V$ is not zero.  Split $V$ as a vector space as $V = \bar V \oplus \ker t$.  The action of $t$ then has form:
  $$ t = \begin{pmatrix} \bar t & 0 \\ \tau & 0 \end{pmatrix} $$
  for some $\bar t : \bar V \to \bar V$ and $\tau : \bar V \to \ker t$. We then define $W = V \oplus \ker t = \bar V \oplus \ker t \oplus \ker t$, and give it the locally nilpotent endomorphism
  $$ t = \begin{pmatrix} \bar t & 0 & 0 \\
    \tau & 0 & 0 \\
    0 & \mathrm{id}_{\ker t} & 0 \end{pmatrix}. $$
    The map $W \twoheadrightarrow V$ kills the second copy of $\ker t$.  Any splitting will take $v\in \ker t$ to some element $w\in W$ for which $t(w) = (0,0,v) \in \bar V \oplus \ker t \oplus \ker t$.  Thus no splitting is $t$-linear, and $V$ is not projective.
    
    Note that $\bK[x]$ is the Hopf algebra of functions on the (non-reductive) additive group $\mathbb{G}_a$.  Thus $\cM^{\bK[x]} = \QCoh\bigl((\Spec \bK)/\bG_{a}\bigr)$, and so \Cref{thm.stacks} fails without the virtual linear reducibility requirement.
\end{example}

\begin{lemma}\label{lemma.essentialimage}
  Let $C$ be a coassociative coalgebra and $\{C_i\}_{i\in \cI}$ the partially ordered set of finite-dimensional sub-coalgebras of $C$. Let $\cD$ be a locally presentable linear category. The canonical functor $\cM^C\boxtimes\cD^*  \to \Hom(\cD,\cM^C)$ is fully faithful. Its essential image consists of those functors $F$ occurring as colimits $F = \varinjlim_i F_i$ where $F_i : \cD \to \cM^C$ factors through the inclusion $\cM^{C_i} \hookrightarrow \cM^C$.
\end{lemma}

\begin{proof}
  The hom-tensor adjunction of \Cref{lemma.homtensoradjunction} implies $\boxtimes$ distributes over colimits.  Thus there are canonical equivalences:
  $$ \cM^C\boxtimes\cD^*  \simeq (\varinjlim_i \cM^{C_i})\boxtimes\cD^* \simeq \varinjlim_i \bigl( \cM^{C_i}\boxtimes\cD^*  \bigr) \simeq \varinjlim_i \Hom(\cD,\cM^{C_i})$$
  In the last step we used dualizability of $\cM^{C_i} \cong \cM_{C_i^*}$ to imply that the canonical functor from $\cM^{C_i}\boxtimes\cD^* $ to $\Hom(\cD,\cM^{C_i})$ induced by the pairing is an equivalence: by \cref{pr.dualizable} part (3) and the hom-tensor adunction, 
  \begin{multline*}
    \cM^{C_{i}}\boxtimes \cD^{*} = 
    \cM^{C_{i}}\boxtimes \Hom(\cD,\Vect) \simeq \Hom\bigl( (\cM^{C_{i}})^{*},  \Hom(\cD,\Vect)\bigr) \\ \simeq \Hom\bigl( (\cM^{C_{i}})^{*} \boxtimes \cD, \Vect \bigr)  \simeq \Hom\bigl( \cD,\Hom\bigl((\cM^{C_{i}})^{*},\Vect\bigr)\bigr) \simeq \Hom\bigl(\cD, \cM^{C_{i}}\bigr) .
  \end{multline*}
  
    Since the pairing-induced functor $\cC\boxtimes\cD^*  \to \Hom(\cD,\cC)$ is natural in $\cC$, we conclude that the inclusion
  $$ \varinjlim_i \Hom(\cD,\cM^{C_i}) \simeq  \cM^C\boxtimes\cD^*  \to \Hom(\cD,\cM^C) $$
  is the one induced by the inclusions $\cM^{C_i} \hookrightarrow \cM^C$.  \Cref{lemma.2limitsandcolimits,rem.lim}  it complete the proof.
\end{proof}

We are now equipped to prove \Cref{thm.coalgebras-dualizability}, which asserts that $\cM^C$ is dualizable if and only if it has enough projectives.

\begin{proof}[Proof of \Cref{thm.coalgebras-dualizability}]
  According to~\cite{MR0498663}, the category $\cM^C$ has enough projectives if and only if every finite-dimensional right $C$-comodule has a projective cover; inspection of the proof reveals the attested projective cover to be finite-dimensional.  The finite-dimensional right $C$-comodules are precisely the compact ones, and so we see that $\cM^C$ has enough projectives if and only if it is generated by its compact projective objects. \Cref{cor.dualizability} then implies one direction of the claim.
  
  Suppose now that $\cM^C$ does not have enough projectives.  Then, again by~\cite{MR0498663}, there exists a simple left $C$-comodule $S$ such that the injective hull of $S$ is infinite-dimensional.  It follows that we can find essential extensions $S \hookrightarrow T$ with $\dim T$ finite but arbitrarily large (as otherwise there would be a maximal such $T$, which would therefore be injective).  Recall that an extension $S \hookrightarrow T$ is \define{essential} if any nonzero subobject of $T$ intersects $S$ nontrivially; the dual notion (for abelian categories) is an \define{essential projection} $Q \twoheadrightarrow P$, which is a surjection for which every proper subobject of $Q$ fails to surject onto $P$.  Thus, by dualization, we have found a simple right $C$-comodule $S^* \in \cM^C$ with essential projections $T^* \twoheadrightarrow S^*$ of arbitrarily large dimension.
  
  Let $F\in \Hom(\cM^C,\cM^C)$ be in the essential image of $(\cM^C)^*\boxtimes \cM^C$.  By \Cref{lemma.essentialimage}, $F = \varinjlim F_i$, where $F_i$ is the largest subfunctor of $F$ factoring through $\cM^{C_i}$.  We will prove that $F \not\cong \operatorname{id}_{\cM^C}$.  To do so, consider an arbitrary natural transformation $\theta : F \to \operatorname{id}_{\cM^C}$, or, what is equivalent, a system of natural transformations $\theta_i : F_i \to \operatorname{id}_{\cM^C}$. 
  Since $ \varinjlim (F_i(X)) =  (\varinjlim F_i)(X)$ for all $X \in \cM^C$, it suffices to prove that $\theta_i(S^*) : F_i(S^*) \to S^*$ vanishes for all sufficiently large $i$ and for $S^*$ the simple with arbitrarily large essential surjections from the previous paragraph.  Since $\dim(S^*) < \infty$, for all sufficiently large $i$ we have $S^* \in \cM^{C_i}$, and it suffices to consider just these.
  
  Thus fix $i\in \cI$ with $S^* \in \cM^{C_i} \cong \cM_{C_i^*}$.  Since $\dim(C_i^*) < \infty$, there are bounds on the dimensions of essential surjections onto $S^*$ in $\cM^{C_i}$: a projective cover  is as large as you can get.  We can therefore choose $T^*$ mapping essentially onto $S^*$ so large that $T^* \not\in \cM^{C_i}$.   Consider the following commutative diagram:
  $$ \begin{tikzpicture}[auto]
  \path[anchor=base] (0,0) node (FS) {$F_i(S^*)$} (0,1.5) node (FT) {$F_i(T^*)$} (3,0) node (S) {$S^*$} (3,1.5) node(T) {$T^*$};
  \draw[->>] (FT) -- (FS); \draw[->>] (T) -- (S);
  \draw[->] (FT.mid -| FT.east) -- node {$\theta_i(T^*)$} (T.mid -| T.west);
  \draw[->] (FS.mid -| FS.east) -- node {$\theta_i(S^*)$} (S.mid -| S.west);
  \end{tikzpicture} $$
  The left arrow is a surjection because $F_i$ is cocontinuous.   Since $F_i(T^*) \in \cM^{C_i}$ but $T^* \not\in \cM^{C_i}$, the image of $\theta_i(T^*)$ must be a proper sub-co-module of $T^*$.  (Indeed, it is within the largest $C_i$-subcomodule of $T^*$.)  Since the surjection $T^* \to S^*$ is essential, the composition $F_i(T^*) \to S^*$ cannot be a surjection; since $S^*$ is simple, the composition must vanish.  But $F_i(T^*) \to F_i(S^*)$ is a surjection; hence $\theta_i(S^*) = 0$.
\end{proof}

A similar argument works for $\QCoh(X)$ when $X$ is a projective scheme over $\bK$, and provides the basis of the proof of 
\Cref{thm.projective}. Recall that $X$ is a \define{projective scheme} if it is embeddable as a closed subscheme into some projective space $\bP^N_\bK$.

First, we specialize \Cref{cor.dualizable_alt_criterion} as follows.

\begin{corollary}\label{cor.closed_subscheme}
If a $\bK$-scheme $X$ is such that $\QCoh(X)$ is dualizable, then the same is true of all closed subschemes $i:Y\subseteq X$. 
\end{corollary}
\begin{proof}
Indeed, setting $\cC=\QCoh(X)$, $\cD=\QCoh(Y)$, $\pi=i^*$ and $\iota=i_*$ places us within the scope of \Cref{cor.dualizable_alt_criterion}.
\end{proof}

\begin{proof}[Proof of \Cref{thm.projective}]
If $\QCoh(X)$ is dualizable over $\Vect_\bK$ then, by \Cref{remark.basechange_preserves_dualizability}, for any field extension $\bK\to\bL$, $\QCoh(X\times\Spec(\bL))\simeq\QCoh(X)\boxtimes\Vect_\bL$ is dualizable over $\Vect_\bL$. Consequently, we may as well assume $\bK$ is algebraically closed. By \Cref{cor.closed_subscheme} we may also assume that $X$ itself is integral and projective.
 
Any embedding of $X$ into some projective space provides a spanning set of objects $\cO(n) \in \QCoh(X)$ such that  there are no non-zero maps  $\cO(m)\to \cO(n)$ for $m>n$. This simply says that $\cO(-n)$ has no (non-zero) sections if $n>0$; in fact, if there were such sections, then similarly $\cO(-kn) \cong \cO(-n)^{\otimes k}$ would have non-zero sections for any $k>0$. The condition $\dim(X)\ge 1$ ensures that when $k$ is large then $\cO(kn)$ has a large space of sections, in which case so would $\cO\cong \cO(-kn)\otimes\cO(kn)$. But this contradicts the fact that the only global regular functions on a projective variety are the constants (\cite[Theorem I.3.4]{MR0463157}).  

If $\QCoh(X)_{\geq n}$ is the full, cocomplete subcategory of $\QCoh(X)$ spanned by $\{\cO(m)\}_{m \geq n}$, then  $ \QCoh(X) \simeq \varinjlim_{n \to -\infty}  \QCoh(X)_{\geq n}$.  Without dualizability of $\QCoh(X)_{\geq n}$, the essential image of $\QCoh(X)^* \boxtimes \QCoh(X)$ inside $\Hom(\QCoh(X),\QCoh(X))$ may fail to include all colimits of the form $F = \varinjlim F_{n}$ for $F_{n}$ factoring through $\QCoh(X)_{\geq n}$, but the arguments of \Cref{lemma.essentialimage} do imply that any $F$ in the essential image is of this type.  Thus, as in the proof of \Cref{thm.coalgebras-dualizability}, it suffices to find a nonzero object $M \in \QCoh(X)$ such that for all sufficiently negative $n$, any natural transformation $\theta_{n} : F_{n} \to \operatorname{id}_{\QCoh(X)}$ satisfies $\theta_{n}(M) = 0 : F_{n}(M) \to M$.
  
  Fix $k$ arbitrarily and set $M = \cO(k)$.  For any $n \leq k$, we can find some direct sum of $\cO(m)$s with $m < n$ that surjects onto $\cO(k)$.   We thus build a commutative square similar to the one from the proof of \Cref{thm.coalgebras-dualizability}:
  $$ \begin{tikzpicture}[auto]
  \path[anchor=base] (0,0) node (FS) {$F_n\bigl(\cO(k)\bigr)$} (0,1.5) node (FT) {$F_n\bigl(\bigoplus \cO(m)\bigr)$} (5,0) node (S) {$\cO(k)$} (5,1.5) node(T) {$\bigoplus \cO(m)$};
  \draw[->>] (FT) -- (FS); \draw[->>] (T) -- (S);
  \draw[->] (FT.mid -| FT.east) -- node {$\theta_n\bigl(\bigoplus \cO(m)\bigr)$} (T.mid -| T.west);
  \draw[->] (FS.mid -| FS.east) -- node {$\theta_n\bigl(\cO(k)\bigr)$} (S.mid -| S.west);
  \end{tikzpicture} $$
  As before, the left arrow is a surjection since $F_n$ is cocontinuous.  But $F_n\bigl(\bigoplus \cO(m)\bigr) \in \QCoh(X)_{\geq n}$, and so $\theta_n\bigl(\bigoplus \cO(m)\bigr) = 0$.  It follows that $\theta_n\bigl(\cO(k)\bigr) = 0$.
\end{proof}


\section{Reflexivity} \label{section.reflexivity}

The goal of this section is to prove \Cref{pr.aux}, which we
 will use  to prove \Cref{thm.reflexivity-stronger,thm.tensoriality}.  The arguments in this section apply when $\bK$ is not a field but just a commutative ring, in which case ``$\Vect$'' means the symmetric monoidal category $\cM_{\bK}$ of all $\bK$-modules.

In this section we denote by $(\cI,\le)$ a generic poset which is $\aleph_{0}$-directed: any two $i,j\in\cI$ are dominated by some $k\in\cI$. As always, we regard $\cI$ as a category with an arrow $i\to j$ for $i\le j$. 

\begin{definition}\label{defn.acceptable}
An \define{$\cI$-indexed pro-object} (or just pro-object, when $\cI$ is understood) in a category $\cT$ is a functor $\cI^\op\to\cT$. 
An \define{$\cI$-indexed pro-algebra} (or just pro-algebra) is an $\cI$-indexed pro-object in $\bK$-algebras. 
An \define{acceptable} ($\cI$-indexed) pro-algebra is a pro-algebra whose indexing poset $\cI$ has countable cofinality (i.e.\ it has a countable cofinal subset), and for which the transition maps $A_j\to A_i$ for $i\le j$ are onto. 
\end{definition}

Note that for any pro-algebra $\{A_i\}_{i\in\cI^\op}$ and $i\le j$ we can pull back $A_{i}$-modules to $A_{j}$-modules, thereby getting an inductive system of locally presentable $\bK$-linear categories $\{\cM_{A_{i}}\}_{i\in \cI}$.

\begin{remark}\label{remark.I_to_N}
The cofinal countability of $\cI$ for acceptable pro-algebras will come up in a number of ways:

First, it allows us to assume that $(\cI,\le)$ is $\bN=\{0,1,\ldots\}$ with the usual order, as pro-algebras indexed by varying posets form a category, and every pro-object whose indexing poset has countable cofinality is isomorphic in this category to an $\bN$-indexed one. Since the discussion below is invariant under isomorphism in the category of pro-algebras, we will make such a substitution in  the sequel. 

Second, for an acceptable $\cI$-indexed pro-algebra $\{A_i\}$ the maps from the limit $\widehat A := \varprojlim_i A_i$ to the individual $A_j$s are onto. This is seen by first turning the pro-algebra into an acceptable $\bN$-indexed one, in which case surjectivity is obvious. In general, when $\cI$ has uncountable cofinality, it is possible for all $A_{i}$s to be nontrivial (i.e.\ not the ground field) and all transition maps $A_{j}\to A_{i}$ to be onto but nevertheless for $\varprojlim A_{i}$ to be trivial; see e.g.~\cite[Corollary 8]{Ber}. 
\end{remark}

\begin{remark}\label{remark.motivationfordiscrete}
One can pull back $A_{i}$-modules to $\widehat A$-modules. These pull back functors are compatible, and therefore give a functor $\varinjlim_{i} \cM_{A_{i}} \to \cM_{\widehat A}$, which is easily seen to be full and faithful.  Its essential image consists of those $\widehat A$-modules $V$ such that for every $v\in V$ there is some $i \in \cI$ for which $v$ is killed by the kernel of the projection $\widehat A \twoheadrightarrow A_{i}$.  More generally, the essential image of 
\[
 (\varinjlim_{i} \cM_{A_{i}}) \boxtimes \cM_{B} \simeq \varinjlim_{i} (\cM_{A_{i}} \boxtimes \cM_{B}) \to \cM_{\widehat A} \boxtimes \cM_{B} \simeq \cM_{\widehat A \otimes B}
\] 
consists of those $(\widehat A\otimes B)$-modules $V$ whose underlying $\widehat{A}$-modules satisfy the property above. This motivates the following definition:
\end{remark}

\begin{definition}\label{defn.discrete}
A module $V$ over a filtered limit $\widehat{A}=\varprojlim_i A_i$ of algebras is called \define{discrete} if for every $v\in V$ there is some $i \in \cI$ for which $v$ is killed by $\ker(\widehat A \twoheadrightarrow A_{i})$.
\end{definition}

Now, since $(-)^{*} = \Hom(-,\Vect)$ turns colimits into limits, we have 
\[
 (\varinjlim \cM_{A_{i}})^{*} \simeq \varprojlim( \cM_{A_{i}})^{*} \simeq \varprojlim ({_{A_{i}}\cM}) 
\] 
(where the second equivalence uses the Eilenberg--Watts theorem).
An object of this limit is a sequence $\{M_{i}\}_{i\in \cI}$ whose $i$th entry is a left $A_{i}$-module, along with isomorphisms $M_{i} \cong A_{i} \otimes_{A_{j}} M_{j}$ for $j>i$, compatible for triples $k>j>i$. Notice that this implies that the maps $M_j \to M_i$ are onto.  An example is the regular module $A = \{A_{i}\}_{i\in\cI}$ with the canonical isomorphisms, which when thought of as an object of $(\varinjlim \cM_{A_{i}})^{*}$ is nothing but the forgetful functor $\varinjlim \cM_{A_{i}} \to \Vect$.  It enjoys $\hom(A,A) \cong \widehat A$.

More generally,  any $M  \in \varprojlim ({_{A_{i}}\cM})$ is among other things a pro-vector space $M = \{M_{i}\}_{i\in \cI^{\op}}$, and $\hom(A,M) \cong \widehat M = \varprojlim M_{i}$ is its limit in $\Vect$. It consists of \define{compatible systems} of elements $\{m_{i}\in M_{i}\}_{i\in \cI}$, i.e.\ $m_{j}$ is the image of $m_{i}$ under the projection $M_{i} \twoheadrightarrow A_{j}\otimes_{A_{i}} M_{i} \cong M_{j}$.

For the remainder of this section we let   $\cM=\varprojlim_i {_{A_i}}\cM$.

\begin{lemma}\label{lemma.presentation_aux}
Let $M=\{M_i\}_{i\in \cI}$ and $N=\{N_i\}_{i\in\cI}$ be objects in $\cM$, and $f:M\to N$ be an epimorphism (i.e.\ all  components $f_i:M_i\to N_i$ are onto). Let also $m_j\in\ker(f_j)$ be an element for some fixed $j\in\cI$. Then $m_j$ can be expanded to a compatible system of elements $m_i\in\ker(f_i)$, $i\in\cI$. 
\end{lemma}
\begin{proof}
We may as well assume $\cI = \bN$ by \cref{remark.I_to_N} and $j=0$. Choose any preimage $m'_1 \in M_1$, and consider its image $n'_1 \in N_1$. It lies in $\ker(N_1 \to N_0) = \ker(A_1 \to A_0) \cdot N_1$, which thus lifts to some element in $\ker(A_1 \to A_0) \cdot M_1$. Subtracting this element from $m'_1$, we get our desired lift $m_1 \in M_1$. Now repeat to construct $m_2$, etc.
\end{proof}

\begin{lemma}\label{pr.canonical_presentation}
The objects of $\cM$ admit functorial presentations of the form
\[
  N\mapsto (A^{\oplus S_1(N)}\to A^{\oplus S_0(N)}\to N\to 0),
\]
where $S_i$ are functors $\cM\to\cat{Set}$. 
\end{lemma}
\begin{proof}
Define $S_0(N)=\widehat{N}$, and the natural epimorphism $A^{\oplus S_0(N)}\to N$ in the obvious way: For $n\in\widehat{N}\cong\hom_\cM(A,N)$, the map from the $n$th summand of $A^{\oplus\widehat{N}}$ to $N$ is simply $n$. 

The resulting map $A^{S_0(N)}\to N$ is surjective at each level $j\in\cI$, because any $n_j\in N_j$ can be lifted to a compatible system of elements $n_i\in N_i$, $i\in\cI$. 

Next, let $S_1(N)$ be the set of compatible systems of elements $a_i\in\ker\bigl(A_i^{\oplus S_0(N)}\to N_i\bigr)$ and define $A^{\oplus S_1(N)}\to A^{\oplus S_0(N)}$ similarly as before. Exactness of
\[
 A^{\oplus S_1(N)}\to A^{\oplus S_0(N)}\to N\to 0
\]
follows from \Cref{lemma.presentation_aux} by taking $f$ to be our morphism $A^{\oplus S_0(N)}\to N$. 
\end{proof}

We are now equipped to prove \cref{pr.aux}, of which \Cref{thm.reflexivity-stronger,thm.tensoriality} are consequences.

\begin{proposition}\label{pr.aux}
Let $A=\{ A_{i}\}_{i\in \cI^{\op}}$ be an acceptable pro-algebra and $B$ be any $\bK$-algebra.  The functor
\[ \eta: \Hom \Bigl(\varprojlim ({_{A_{i}}\cM}),\cM_{B} \Bigr)\to \cM_{\widehat A \otimes B},\quad F\mapsto F(A) \]
is fully faithful. Its essential image consists of those $\widehat{A}\otimes B$-modules which are discrete over $\widehat{A}$ in the sense of \Cref{defn.discrete}.
\end{proposition}

\begin{proof} As before, we will abbreviate $\cM := \varprojlim ({_{A_{i}}\cM})$.

\begin{proofstep}{The essential image of $\eta$ contains all $\widehat{A}$-discrete $\widehat{A}\otimes B$-modules.}

Given some $\widehat{A}$-discrete $\widehat{A} \otimes B$-module $V$, let $V_i := \{v \in V : \ker(\widehat{A} \to A_i) \cdot v = 0\}$ be the largest $A_i$-submodule of $V$. Then $i \leq j$ implies $V_i \subseteq V_j$ and we have $V = \varinjlim_i V_i$. If $M \in \cM$, we obtain a $B$-linear map
\[V_i \otimes_{A_i} M_i \cong V_i \otimes_{A_i} (A_i \otimes_{A_j} M_j) \cong V_i \otimes_{A_j} M_j \to V_j \otimes_{A_j} M_j.\]
Define the functor  $F : \cM \to \cM_B$ by $F(M) := \varinjlim_i V_i \otimes_{A_i} M_i$. Clearly, $F$ is cocontinuous and satisfies $F(A) \cong V$.
\end{proofstep}

\begin{proofstep}{The essential image of $\eta$ contains only  $\widehat{A}$-discrete $\widehat{A}\otimes B$-modules.} \label{newstep2}

Let $F \in \Hom\bigl( \cM, \cM_{B}\bigr)$. By \cref{remark.I_to_N} we may assume $\cI = \bN$.  We want to show that for every $v\in F(A)$, there is an $i\in \bN$ such that $v$ is killed by $\ker\bigl(\widehat A \twoheadrightarrow A_{i}\bigr)$.

Consider the direct sum $A^{\oplus \bN}$ of $\aleph_0$-many copies of $A$ in $\cM$.  The object $A$, although a generator of $\cM$, is not compact when the pro-algebra $A$ does not ``stabilize'' (i.e.\ when $A_{i+1}\to A_i$ are not  isomorphisms for infinitely many $i$).  A manifestation of this is that $\hom(A,A^{\oplus \bN})$ is not $\widehat A^{\oplus \bN}$ but rather the projective limit of $A^{\oplus \bN}$ considered as a diagram in $\Vect$.  In other words, a homomorphism $A\to A^{\oplus \bN}$ is the same as a system of homomorphisms $\{f_i:A\to A\}_{i\in \bN}$ with the property that for all $j\in \bN$, all but finitely many of the $f_i$s belong to  $\ker\bigl(\widehat A\twoheadrightarrow A_{j}\bigr)$.

Suppose now that $v\in F(A)$ is such that for every $i\in \bN$, there is some $f_i\in\ker\bigl(\widehat A\to A_i\bigr)$ which does not annihilate $v$. These $f_{i}$s define a homomorphism $f : A \to A^{\oplus \bN}$.  Since $F$ is cocontinuous, we get a linear map
\[
 F(f) : F(A) \to F\bigl(A^{\oplus \bN}\bigr) \cong F(A)^{\oplus \bN}
\]
for which every component of $F(f)(v)$ is non-zero. This is absurd, since $F(f)(v)$ is supposed to sit inside the direct sum $F(A)^{\oplus \bN}$. 
\end{proofstep}

\begin{proofstep}{The functor $\eta$ is faithful.}

This follows immediately from \Cref{pr.canonical_presentation} and the fact that colimit-preserving functors preserve epimorphisms.
\end{proofstep}

\begin{proofstep}{The functor $\eta$ is full.}

Let $F,G\in\Hom(\cM\to\cM_B)$. We have to show that any $\widehat{A}\otimes B$-morphism $\phi:F(A)\to G(A)$ is the $A$-component $\theta_A$ of a natural transformation $\theta:F\to G$. 

Recall from \Cref{lemma.presentation_aux} that we have functorial free resolutions for objects in $\cM$.   We claim that to prove fullness of $\eta$,
 it suffices to check, for all sets $S$ and $T$ and all morphisms $f:A^{\oplus S}\to A^{\oplus T}$, the commutativity of 
\begin{equation}\label{eq.preaux}
  \begin{tikzpicture}[auto,baseline=(current  bounding  box.center)]
    \path[anchor=base] (0,0) node (FAS) {$F(A)^{\oplus S}$} +(3,0) node (FAT) {$F(A)^{\oplus T}$} +(0,-1.5) node (GAS) {$G(A)^{\oplus S}$} + (3,-1.5) node(GAT) {$G(A)^{\oplus T}$};
    \draw[->] (FAS) to node[auto] {$\scriptstyle F(f)$} (FAT); 
    \draw[->] (GAS) to node[auto] {$\scriptstyle G(f)$} (GAT);
    \draw[->] (FAS) to node[auto,swap] {$\scriptstyle\phi^{\oplus S}$} (GAS);
    \draw[->] (FAT) to node[auto] {$\scriptstyle\phi^{\oplus T}$} (GAT);
  \end{tikzpicture}
\end{equation}  
Indeed, we can then set $\theta_{A^{\oplus S}}=\phi^{\oplus S}$ for every set $S$ and extend this to a natural transformation $\theta$ via the functorial resolutions. 

The right hand vertical arrow in \Cref{eq.preaux} embeds into
\begin{equation*}
  \begin{tikzpicture}[auto,baseline=(current  bounding  box.center)]
    \path[anchor=base] (0,0) node (FAT) {$ F(A)^{\times T}$} +(3,0) node(GAT) {$G(A)^{\times T}$};
    \draw[->] (FAT) to node[auto] {$\scriptstyle \phi^{\times T}$} (GAT); 
  \end{tikzpicture}
\end{equation*}  
(direct product rather than direct sum).  For each $s\in S$ and each $t\in T$, let $\iota_{s} :A \to A^{\oplus S}$ denote the inclusions onto the $s$th summand, and $\pi_{t} : A^{\oplus T}\to A$ the projection onto the $t$th summand (the latter exists in any additive category).  Then the $(t,s)$th matrix entry in the composition $F(A)^{\oplus S} \overset{F(f)}\longrightarrow F(A)^{\oplus T} \to F(A)^{\times T}$ is just $F(f')$ for $f' = \pi_{t}\circ f \circ \iota_{s}$; the $(t,s)$th matrix entry in $G(A)^{\oplus S}\to G(A)^{\times T}$ is $G(f')$ for the same $f' \in \hom_{\cM}(A,A)$.
Thus, by replacing $f$ with $f'$ in \cref{eq.preaux}, we can assume $S$ and $T$ are singletons.

Finally, the commutativity of  
\begin{equation*}\label{eq.aux}
  \begin{tikzpicture}[auto,baseline=(current  bounding  box.center)]
    \path[anchor=base] (0,0) node (FA) {$F(A)$} +(3,0) node (FAT) {$F(A)$} +(0,-1.5) node (GA) {$G(A)$} + (3,-1.5) node(GAT) {$G(A)$};
    \draw[->] (FA) to node[auto] {$\scriptstyle F(f')$} (FAT); 
    \draw[->] (GA) to node[auto] {$\scriptstyle G(f')$} (GAT);
    \draw[->] (FA) to node[auto,swap] {$\scriptstyle\phi$} (GA);
    \draw[->] (FAT) to node[auto] {$\scriptstyle\phi$} (GAT);
  \end{tikzpicture}
\end{equation*}
follows from the fact that $F(f')$ and $G(f')$ are nothing but the actions of $f' \in \hom_{\cM}(A,A) \cong \widehat A$ on $F(A)$ and $G(A)$,
while $\phi$ is by definition a morphism of $\widehat{A}$-modules. 
\end{proofstep}
\end{proof}

\begin{remark}
Actually \Cref{pr.aux} can be proven in a more general setting and provides a partial universal property of $\varprojlim ({_{A_{i}}\cM})$: If $\cC$ is a locally finitely presentable $\bK$-linear category (i.e.\ locally presentable and spanned by its compact objects), then $\Hom \bigl(\varprojlim ({_{A_{i}}\cM}),\cC \bigr)$ is equivalent to the category of discrete $\widehat{A}$-module objects in $\cC$. These are objects $T \in \cC$ equipped with a homomorphism of $\bK$-algebras $\widehat{A} \to \mathrm{End}_{\cC}(T)$ such that $T = \varinjlim_i T_i$, where $T_i := \ker\bigl(T \to \prod_{a \in \ker(\widehat{A} \to A_i)} T\bigr)$ is the largest $A_i$-submodule object of $T$. The same proof as before works. Only in \hyperref[newstep2]{Step \ref{newstep2}} we have to replace $v$ by a morphism $P \to F(A)$, where $P$ is any compact object.
\end{remark}

\begin{proof}[Proof of \Cref{thm.reflexivity-stronger}]
\Cref{lemma.2limitsandcolimits} identifies the category $\cC$ from the statement of \Cref{thm.reflexivity-stronger} with the colimit $\varinjlim_{i}\cM_{A_{i}}$ in $\cat{Pres}_{\bK}$. Consider the commutative triangle 
$$ \begin{tikzpicture}[auto]
  \path node (tensor) {$\bigl(\varinjlim_{i} \cM_{A_{i}}\bigr) \boxtimes \cM_{B}$}
    +(3,-1.5) node (codomain) {$\cM_{\widehat{A}\otimes B}$}
    +(6,0) node (hom) {$\Hom \Bigl(\bigl(\varinjlim_{i} \cM_{A_{i}}\bigr)^{*},\cM_{B} \Bigr)$};
  \draw[->] (tensor) -- node {$\scriptstyle \dd$} (hom);
  \draw[right hook->] (tensor) -- node[swap] {$\scriptstyle \text{inclusion}$} (codomain);
  \draw[->] (hom) -- node {$\scriptstyle \eta$} (codomain);
\end{tikzpicture} $$

of functors. According to \Cref{remark.motivationfordiscrete,pr.aux}, both vertical functors are fully faithful with essential image the category of $\widehat{A}$-discrete $\widehat{A}\otimes B$-modules. It follows that $\dd$ is an equivalence of categories.
\end{proof}

We conclude this section with a proof of \Cref{thm.tensoriality}. Recall that for a stack $X$ over $\bK$, the category $\QCoh(X)$ is not just $\bK$-linear and locally presentable, but also symmetric monoidal in the sense of \Cref{defn.tensor_loc_pres}.
 
Following~\cite{Brandenburg2011,HallRydh2014}, call a stack $X$ over $\bK$ \define{tensorial} if for any affine scheme $\Spec(B)$ over $\bK$, the functor
\[ \hom\bigl(\Spec(B),X\bigr) \to \Hom_{\otimes}\bigl( \QCoh(X),\cM_{B}\bigr), \quad f \mapsto f^{*} \]
is an equivalence. Such stacks are called ``2-affine'' in~\cite{MR3097055,MR3144607}, where symmetric monoidal locally presentable (resp. cocomplete) categories are called ``commutative 2-rings.''  

In particular, we can talk about tensorial schemes, a notion which we now extend to ind-schemes in the obvious manner.  Recall that an ind-scheme $X=\varinjlim_i X_i$ is a formal directed colimit of schemes $X_i$ where the transition morphisms $X_i \to X_j$ ($i \leq j$) are closed immersions. One defines $\QCoh(X) := \varprojlim_{i} \QCoh(X_i)$.

\begin{definition}\label{defn.tensoriality}
An ind-scheme $X=\varinjlim_{i \in I} X_i$ is \define{tensorial} if for every commutative algebra $B$ the canonical functor
$$  \varinjlim_i \hom\bigl(\Spec(B),X_i\bigr) = \hom\bigl(\Spec(B),X\bigr) \to \Hom_{\otimes}\bigl( \QCoh(X),\cM_{B}\bigr), \quad f \mapsto f^{*} $$
is an equivalence of categories. Hence, given an acceptable commutative pro-algebra $\{A_i\}_{i\in\cI^\op}$, the corresponding affine ind-scheme $\varinjlim \Spec(A_{i})$ is tensorial if for every commutative algebra~$B$ the canonical functor
\begin{equation*}\label{eq.tensorial}
\varinjlim_i \hom\bigl(A_i,B\bigr) \to \Hom_{\otimes}\Bigr( \varprojlim_i \cM_{A_i},\cM_{B}\Bigl)
\end{equation*}
is an equivalence of categories.
\end{definition}

\begin{proof}[Proof of \Cref{thm.tensoriality}]
Let $\{A_i\}_{i\in\cI^\op}$ and $B$ be as above. Let $\cM := \varprojlim_i \cM_{A_i}$. By \Cref{pr.aux} the category $\Hom_\otimes(\cM,\cM_B)$ is equivalent to the category of $\widehat{A}$-discrete $\widehat{A} \otimes B$-modules $V$ together with a symmetric monoidal structure on the corresponding cocontinuous functor $F : \cM \to \cM_B$ with $F(A)=V$. We will argue that this is equivalent to the (discrete) category of algebra homomorphisms $\widehat{A} \to B$ which factor through some $A_i$.

Symmetric monoidality provides a distinguished isomorphism of $B$-modules $F(A) \cong B$, so that we may assume that the underlying $B$-module of $V$ is just $B$. The $\widehat{A}$-module structure on $B$ comes from the algebra homomorphism $\widehat{A} = \mathrm{End}(A) \to \mathrm{End}(F(A)) \cong \mathrm{End}(B) \cong B$. Discreteness means that every element of $B$ is killed by some $\ker(\widehat{A} \to A_i)$, but it clearly suffices to demand this for $1 \in B$. In other words, $\widehat{A} \to B$ factors through some $A_i$. Next, a morphism $F \to G$ of cocontinuous symmetric monoidal functors $\cM \to \cM_B$ is unique, if it exists, and in that case an isomorphism, because it is completely determined by $F(A) \to G(A)$ (\Cref{pr.aux}) and this has to be the composition of the distinguished isomorphisms $F(A) \cong B$ and $B \cong G(A)$. It remains to remark that every homomorphism $\widehat{A} \to B$ which factors through some $A_i$ is induced by a cocontinuous symmetric monoidal functor $\cM \to \cM_B$, namely $\cM \to \cM_{A_i} \to \cM_{B}$.
\end{proof}

\begin{corollary}
The functor from acceptable commutative pro-algebras to symmetric monoidal locally presentable categories, mapping $\{A_i\}$ to $\varprojlim_i \cM_{A_i}$, is fully faithful.
\end{corollary}

Following~\cite{MR3097055,MR3144607}, this may be interpreted as the statement that affine ind-schemes (indexed by countable posets) are 2-affine.

\begin{proof}
If $A = \{A_i\}_{i\in\cI^{\op}}$ and $B = \{B_j\}_{j\in \cJ^{\op}}$ are acceptable commutative pro-algebras then
\begin{equation*}
\hom\bigl(A,B\bigr) := \varprojlim_j \varinjlim_i \hom(A_i,B_j)  \simeq \varprojlim_{j} \Hom_{\otimes}\Bigr(\! \varprojlim_i \cM_{A_i},  \cM_{B_j}\! \Bigl)  \simeq \Hom_{\otimes}\Bigr(\! \varprojlim_i \cM_{A_i},\varprojlim_j  \cM_{B_j}\! \Bigl). \end{equation*}
The first equivalence uses \Cref{thm.tensoriality}.
\end{proof}


\section{Acknowledgements}

TJF is supported by the NSF grant DMS-1304054.

\bibliographystyle{alphaabbr}

\end{document}